\documentclass[10pt]{amsart}
\usepackage{amssymb,amsmath}
\usepackage{enumerate}
\usepackage{graphicx}
\usepackage{color}
\usepackage{mathrsfs}
\usepackage[numbers]{natbib}

\usepackage[margin=1.2in]{geometry}
\newtheorem{theorem}{Theorem}
\newtheorem{proposition}[theorem]{Proposition}
\newtheorem{lemma}[theorem]{Lemma}
\newtheorem{corollary}[theorem]{Corollary}
\theoremstyle{definition}

\newtheorem{remark}[theorem]{Remark}

\newcommand{\calL}{\mathcal{L}}
\newcommand{\Lpq}{L^{q}_p}
\newcommand{\Hpq}{\mathbb{H}^{2,q}_p([0,T])}

\usepackage{accents}

\begin{document}
\title{Quadratic transportation inequalities for SDEs with measurable drift} 
\author{Khaled Bahlali \and Soufiane Mouchtabih \and Ludovic Tangpi} 
\date{\today}

\begin{abstract}
	Let $X$ be the solution of the multidimensional stochastic differential equation
	\begin{equation*}
		dX(t) = b(t, X(t))\,dt + \sigma(t, X(t))\,dW(t)\, \quad X(0)=x \in \mathbb{R}^d
	\end{equation*}
	where $W$ is a standard Brownian motion.
	In our main result we show that when $b$ is measurable and $\sigma$ is in an appropriate Sobolev space, the law of $X$ satisfies a uniform quadratic transportation inequality.
\end{abstract}

\maketitle
\setcounter{equation}{0}


\section{Introduction and main results}
Throughout this work, we fix $d \in \mathbb{N} $ and $T$ a strictly positive real number.
Let $(\Omega, \mathcal{F}, P)$ be the canonical space of a $d$-dimensional Brownian motion denoted by $W$ and equipped with the $P$-completion of the raw filtration $\sigma(W_s, s\le t)$ generated by $W$.
That is $\Omega=C([0,T],\mathbb{R}^d)$ endowed with the supremum norm, and $W_t(\omega)=\omega(t)$.
Further denote by $\mathcal{P}(\Omega)$ the set of all Borel  probability measures on $\Omega$.
For $\mu,\nu\in\mathcal{P}(\Omega)$ define the (second order) Wassertein distance and the Kullback-Leibler divergence respectively by
\begin{equation*}
	\mathcal{W}_2(\mu,\nu) := \Big(\inf_\pi \int_{\Omega\times\Omega} \|\omega-\eta\|_\infty^2 \,\pi(d\omega,\eta)\Big)^{1/2}
	\quad\text{and}\quad
	H(\nu|\mu) :=\int\frac{d\nu}{d\mu}\log\frac{d\nu}{d\mu}\,d\mu
\end{equation*}
where the infimum is taken over all, probability measures $\pi$ on the product with first marginal $\mu$ and second marginal $\nu$, and we used the convention $d\nu/d\mu=+\infty$ if $\nu$ is not absolutely continuous w.r.t.~$\mu$.
Given a constant $C$, a probability measure $\mu$ is said to satisfy Talagrand's $T_2(C)$ inequality (or quadratic transportation inequality) if
\[
\mathcal{W}_2(\mu,\nu)\leq \sqrt{C H(\nu|\mu)} \text{ for all } \nu\in\mathcal{P}(\Omega).\]
This inequality was popularized in probability theory by the works of \citet{Talagrand96} and \citet{Marton96} on the concentration of measure phenomenon.
It has since found numerous applications, for instance to isoperimetric problems, to randomized algorithms \cite{Dubh-Panc}, or to quantitative finance \cite{ConcenRM,Lac-concent} and to various problems of probability in high dimensions \cite{Del-Lac-Ram_Concent,Massart,backward-chaos}.
We refer the reader e.g. to \citet{Ledoux01} for an overview, notably for the connection to the concentration of measures.
Transportation inequalities are also related to various other functional inequalities as Poincar\'e inequality, log-Sobolev inequality, inf-convolution and hypercontractivity, see \cite{Bob-Gen-Led01}, \cite{Otto-Vil00}.

Our objective is to investigate transportation inequalities for stochastic differential equations of the form
\begin{equation}
\label{eq:SDE intro}
	X(t) = x + \int_0^tb(s,X(s))\,ds + \int_0^t\sigma(s,X(s))\,dW(s)
	\quad\text{for } t\in[0,T], \,\, x \in \mathbb{R}^d.
\end{equation} 
under minimal regularity assumptions on the coefficients $b:[0,T]\times \mathbb{R}^d \to \mathbb{R}^d$ and $\sigma :[0,T]\times \mathbb{R}^d \to \mathbb{R}^{d\times d}$.
To state our main result, let us recall the following functional spaces.
For $p\ge1$ denote by $L^p_{\mathrm{loc}}([0,T]): = L^p_{\mathrm{loc}}([0,T]\times \mathbb{R}^d)$ the (Lebesgue) space of classes of locally integrable functions and for every $m_1 , m_2 \in \mathbb{N}$, let $W^{m_1, m_2}_p([0,T]):= W^{m_1, m_2}_p([0,T]\times \mathbb{R}^d)$ be the usual Sobolev space of weakly differentiable functions $f:[0,T]\times \mathbb{R}^d\to \mathbb{R}$ such that
\begin{equation*}
	\|f\|_{W^{m_1,m_2}_p}:=
	 \sum_{|\alpha|\le m_1}\|\partial^\alpha_t f\|_{L^p} + \sum_{|\alpha|\le m_2}\|\partial^\alpha_x f\|_{L^p} <\infty
\end{equation*}
where $\alpha$ is a multiindex. 
Denote by
$W^{m_1, m_2}_{p,\mathrm{loc}}([0,T])$ the space of weakly differentiable functions $f:[0,T]\times \mathbb{R}^d\to \mathbb{R}$ such that
\begin{equation*}
	\|f\|_{L^p_{\mathrm{loc}}} +
	 \sum_{|\alpha|\le m_1}\|\partial^\alpha_t f\|_{L^p_{\mathrm{loc}}} + \sum_{|\alpha|\le m_2}\|\partial^\alpha_x f\|_{L^p_{\mathrm{loc}}} <\infty.
\end{equation*}
Further let
$L^p_q([0,T]):= L^q([0,T], L^p(\mathbb{R}^d))$ be the space of (classes of) measurable functions $f:[0,T]\times \mathbb{R}^d \to \mathbb{R}^d$ such that
\begin{equation*}
	\|f\|_{L^q_p}:= \Big(\int_0^T\Big(\int_{\mathbb{R}^d}|f(s,x)|^p\,dx \Big)^{q/p}\,ds \Big)^{1/q}<\infty.
\end{equation*}
The aim of this note is to prove the following:
\begin{theorem}
\label{thm:t2 sde}
	Assume that one of the following sets of assumptions is satisfied:
	\begin{itemize}
		\item[(A)] $\sigma, b \in L^\infty([0,T]\times\mathbb{R}^d)$, the function $\sigma$ is continuous in $(t, x)$ and belongs to $W^{0,1}_{2(d+1),\mathrm{loc}}([0,T])$, there is $\lambda >0$ such that 
		\begin{equation}
		\label{eq:elliptic}
		\xi^*\sigma(t,x)\xi \ge \lambda |\xi|^2\quad \text{for all} \quad (t,x,\xi) \in [0,T]\times \mathbb{R}^d\times \mathbb{R}^d, \ \hbox{where $^*$ denotes the transpose, }
		\end{equation} 
	 and $T$ is small enough. 
		\vskip 0.2cm 
		\item[(B)] 	 $\sigma \in W^{0,1}_{2(d+1),\mathrm{loc}}([0,T]) \cap L^\infty([0,T]\times \mathbb{R}^d)$, 
		$\sigma$ is uniformly continuous in $x$ and there is $\lambda>0$ such that 
		$\sigma$ satisfies \eqref{eq:elliptic}. 
		The function $b$ satisfies $b \in L^p_q([0,T])$ for some $p,q$ such that $d/p + 2/q<1$, $2(d+1)\le p$ and $q>2$. 
	\end{itemize}
	Then, Equation \eqref{eq:SDE intro} admits a unique strong solution $X$ with continuous paths and
	\begin{equation*}
		\text{the law $\mu_x$ of $X$ satisfies $T_2(C)$} 
	\end{equation*}
	for some constant $C$ depending on the data, namely $\|b\|_{L^q_p},\|\sigma\|_\infty, T, x, d,p$ and $q$. 
\end{theorem}
Since $b$ is only assumed to be measurable, this result gives transportation inequality for \emph{singular} SDEs as $dX(t) = sgn(X(t))\,dt + \,dW(t)$, or for "regime switching" models as 
$$dX(t) = \big\{ b_1(t,X(t))1_A(X(t) ) + b_2(t,X(t))1_{ A^c}(X(t) \big\}\,dt + \sigma(t,X(t))\,dW(t) $$
with $A$ a measurable subset of $\mathbb{R}^d$. Other examples are discussed at the end of the article.

Regarding the related literature, \citet{Talagrand96} proved a quadratic transportation inequality for the multidimensional Gaussian distribution with optimal constant $C = 2$.
Using stochastic analysis techniques, notably Girsanov's theorem, Talagrand's work was then extended to Wiener measure on the path space by \citet{Fey-Ust04}.
The case of SDEs was first analyzed by \citet{Dje-Gui-Wu} using a technique based on Girsanov's transform that we also employ here.
Their results gave rise to an interesting literature, including the papers
\cite{Pal12, Uestuenel12} on SDEs driven by Brownian motion and \cite{Riedel, Sau} on SDEs driven by abstract Gaussian noise.
Note that all the aforementioned works on SDEs assume that the coefficients are Lipschitz-continuous or satisfy a dissipative condition.

The effort to extend results from \cite{Dje-Gui-Wu} to diffusions with non-smooth coefficients was started by \citet{t2bsde} where it is proved that $T_2(C)$ holds for \emph{one-dimensional} equations, if $b$ is measurable in space and \emph{differentiable} in time and $\sigma$ \emph{Lipschitz} continuous.
The idea of \cite{t2bsde} is based on a transformation that is tailor-made for the one-dimensional case.
The present note deals with the multidimensional case and further weakens the regularity requirements imposed in \cite{t2bsde}.
In this case we use Zvonkin's transformation, (a technique well-known in SDE theory) along with gradient estimates for singular, second order parabolic PDEs.
Note that considering multidimensional equations is, for instance, fundamental for applications to concentration and asymptotic results on \emph{interacting particle systems}, see e.g. \cite[Section 5]{Del-Lac-Ram_Concent} and the various examples we give in the final section.

The proof of Theorem \ref{thm:t2 sde} is given in the next section, and the final section presents some examples.

\section{Transportation inequalities}

\subsection{Equation with Sobolev coefficients.}
The goal of this section is to prove a quadratic transportation inequality for SDE \eqref{eq:SDE intro} when the coefficients belong to some Sobolev spaces. 
Along with gradient estimates for solutions of singular PDEs, this will be an essential building block for the proof of the main result.

\begin{proposition} 
\label{prop:sde no drift}
	Let $\sigma \in W^{0,1}_{2(d+1),\mathrm{loc}}([0,T])\cap L^\infty([0,T]\times \mathbb{R}^d)$ and $b \in W^{0,1}_{(d+1),\mathrm{loc}}([0,T])\cap L^\infty([0,T]\times \mathbb{R}^d)$.
Assume that there exists $\lambda >0$ such that $\sigma$ satisfies \eqref{eq:elliptic}. Then, equation \eqref{eq:SDE intro} admits a unique strong solution $X$, and
	\begin{equation*}
		\text{the law $\mu_x$ of $X$ satisfies $T_2(C)$} 
	\end{equation*}
	with $\displaystyle C = \inf_{0<\varepsilon<1}2\exp\Big(6\frac{C^2_{BDG}+\varepsilon}{\varepsilon(1 - \varepsilon)}T \Big)\frac{1}{1-\varepsilon}\|\sigma\|_\infty^2$ where $C_{BDG}$ is the universal constant appearing in Burkholder-Davis-Gundy inequality.
\end{proposition} 

The proof of this proposition follows a coupling argument introduced in \cite{Dje-Gui-Wu}.
The main challenge here being the lack of regularity of the coefficients $\sigma$ and $b$.
We start by a Lemma whose proof can be found in Step 1 of the proof of \cite[Theorem 5.6]{Dje-Gui-Wu}.
\begin{lemma}
\label{lem:girsanov}
	Let $\nu \in \mathcal{P}(\Omega)$ be such that $\nu \ll \mu_x$ and $H(\nu|\mu_x)<\infty$, and let $X$ be the solution of \eqref{eq:SDE intro}.
	Then, the probability measure $\nu$ given by 
	\begin{equation*}
		Q:= \frac{d\nu}{d\mu_x}(X)P
	\end{equation*}
	satisfies 
	\begin{equation}
	\label{eq:identity H}
		H(\nu|\mu_x) = E_Q\Big[\frac12\int_0^T|q(s)|^2\,ds \Big]
	\end{equation}
	for some progressively measurable, square integrable process $q$ such that
	$\tilde W := W-\int_0^\cdot q(s)\,ds$ is a $Q$-Brownian motion.
\end{lemma}
\begin{proof}[Proof of Proposition \ref{prop:sde no drift}]  
	That \eqref{eq:SDE intro} admits a unique strong solution follows from e.g. \cite[Theorem 2.1]{Bah99}.   
	Let $\nu \in \mathcal{P}(\Omega)$ be absolutely continuous with respect to $\mu$.
	We can assume without loss of generality that $H(\nu|\mu_x)<\infty$.
	Let $Q$ and $q$ be as in Lemma \ref{lem:girsanov}.
	Under the probability measure $Q$, the SDE \eqref{eq:SDE intro} takes the form
	\begin{equation}
	\label{eq:sde no drift q}
		dX(t)=\sigma(t,X(t))d\tilde{W}(t)+  \big\{ \sigma(t,X(t))q(t) + b(t, X(t)) \big\}dt,\quad \text{with} \quad X(0)=x
	\end{equation}
	and the law of $X$ under $Q$ is $\nu$.
	Furthermore, the SDE
	\begin{equation*}
		dY(t)=\sigma(t,Y(t))d\tilde{W}(t) + b(t, Y(t))\,dt,\quad \text{with } Y(0) =x
	\end{equation*}
	admits a unique solution whose law under $Q$ is $\mu_x$.
	That is, $(X,Y)$ under $Q$ is a coupling of $(\nu, \mu_x)$.
	Thus,
	\begin{equation}
	\label{eq:first estim wasser}
		\mathcal{W}^2_2(\nu, \mu_x) \le  E_{Q}\Big[\sup_{0\le t\le T}|X(t)-Y(t)|^2\Big].	
	\end{equation}

	We now estimate the right hand side above.
	By It\^o's formula, we have
	\begin{align}
	\nonumber
		|X(t) - Y(t)|^2 &= \int_0^t2(X(s) - Y(s))\sigma(s, X(s))q(s) + |\sigma(s, X(s)) - \sigma(s, Y(s))|^2\,ds\\\notag
		&\quad  + \int_0^t2(X(s) - Y(s))(b(s,X(s)) - b(s,Y(s)))\,ds\\
		&\quad +  \int_0^t2(X(s) - Y(s))(\sigma(s, X(s)) - \sigma(s, Y(s)))\,d\tilde W,
		\label{eq:first ito}
	\end{align}
	where we simply denote by $ab$ the inner product between two vectors $a$ and $b$.
	The difficulty is to deal with the terms $\sigma(s, X(s)) - \sigma(s, Y(s))$ and $b(s, X(s)) - b(s, Y(s))$.
	To that end, we introduce the following random times:
	First consider the sequence of stopping times
	\begin{equation*}
		\tau^N:= \inf\{ t>0: |X(t)|>N \text{ or } |Y(t)|>N\}\wedge T.
	\end{equation*}
	It is clear that $\tau^N\uparrow T$. 
	For each $\lambda$ in $[0, 1]$ and $t$ in $[0, T]$, we put $Z_t^{\lambda}:= \lambda X(t) + (1-\lambda) Y(t)$. 
	For every $N \ge 0$ and $ t\in [0,\infty)$, define 
	\begin{equation*}
		A^N(t) := \int_0^{t \wedge \tau^N}\int_{0}^{1}(|\partial_x\sigma(s, Z_s^{\lambda})|^2+|\partial_xb(s,Z_s^{\lambda})|)  d\lambda ds, 
	\end{equation*}
	where the (weak) derivative acts on the spacial variable. 
	The process $k^N(t):= t + A^N(t)$ is continuous, strictly increasing and $k^N(0) = 0$. 
	Moreover, $k^N$ maps $[0,\infty)$ onto itself. 
	We denote by $\gamma^N$ the unique inverse map of $k^N$.
	Using \eqref{eq:first ito}, Cauchy-Schwarz and Burkholder-Davis-Gundy inequalities, one can show that for each $t\in [0,\infty)$ it holds that  
	\begin{align}
	\nonumber
		&E_Q\Big[\sup_{s \in [0,\gamma^N_t\wedge\tau^N]}|X(s)-Y(s)|^2\Big] \\\notag
		&\quad\le  E_Q\Big[\int_0^{\gamma^N_t\wedge\tau^N}|X(s)-Y(s)|^2\, ds+ \|\sigma\|_\infty^2 \int_0^{\gamma^N_t\wedge\tau^N}|q(s)|^2ds\Big]\\\notag
		&\qquad+E_Q\Big[\int_0^{\gamma^N_t\wedge\tau^N}|\sigma(s,X(s))-\sigma(s,Y(s))|^2 +2(X(s) - Y(s))(b(s,X(s)) - b(s,Y(s)))\,ds \Big]\\\notag
		&\qquad +2C_{BDG}E_Q\Big[\Big(\int_{0}^{\gamma^N_t\wedge\tau^N}|X(s)-Y(s)|^2|\sigma(s,X(s))-\sigma(s,Y(s))|^2\,ds\Big)^{1/2}\Big]\\
		\label{eq:first diff estimate}
		&\quad := I_1 + I_2 + I_3,
	\end{align}
	for a (universal) constant $C_{BDG}>0$. 
	
	Let $\varepsilon>0$. By Young's inequality, we have
	\begin{align*}
		I_3 \le \varepsilon E_Q\Big[\sup_{s \in [0, \gamma^N_t\wedge \tau^N]}|X(s) - Y(s) |^2 \Big] + \frac{2C^2_{BDG}}{\varepsilon}I_2.
	\end{align*} 
	We shall estimate $I_2$. 
	 We denote by $K_N$  the ball $\{x\in\mathbb{R}^d, |x|\le N\}$. 
	For  $K:=[0,T]\times K_N$.  
	Let $b_n,\sigma_n\in C^\infty$ be such that, 
	 \begin{equation*}
	 	\|\sigma_n - \sigma\|_{W^{0,1}_{2(d+1)}(K)} \to 0 \quad \text{and} \quad \|b_n - b\|_{W^{0,1}_{d+1}(K)} \to 0.
	 \end{equation*} 
	 Using Krylov's estimate (\cite[Theorem 2.2.4]{Krylov}), we get  
	\begin{align*} 
		I_2  &\le  
		C_{T,N,d}\,\|\sigma-\sigma_n\|_{L^{2(d+1)}(K)} + C_{T,N,d}\,(\|b-b_n\|_{L^{(d+1)}(K)})\\
		&\quad + 2E_{Q}\Big[\int_0^{\gamma_t^N\wedge\tau^N}|X(s)-Y(s)|^2\int_{0}^{1}(|\partial_x\sigma_{n}(s, Z_s^{\lambda})|^2+|\partial_xb_n(s,Z_s^{\lambda})|)  d\lambda ds\Big]
	\end{align*} 
	where   $C_{T,N,d}$ is a positive constant which depends on $T, N$ and $d$. 
	Taking the limit as $n$ goes to infinity in the last inequality and using the fact that $dA^N(s) \leq dk^N(s$), we obtain
	\begin{align*}
		I_2 \le 2 E_Q\Big[\int_0^{\gamma_t^N\wedge\tau^N}|X(s)-Y(s)|^2\,dk^N(s) \Big].
	\end{align*}
	Therefore,
	\begin{align*}
		I_3 \le \varepsilon E_Q\Big[\sup_{s \in [0, \gamma^N_t\wedge \tau^N]}|X(s) - Y(s) |^2 \Big] + \frac{6C_{BDG}^2}{\varepsilon}E_Q\Big[\int_0^{\gamma_t^N\wedge\tau^N}|X(s)-Y(s)|^2\,dk^N(s) \Big].
	\end{align*}
	Coming back to \eqref{eq:first diff estimate}, since $k^N(t) := t + A^N(t)$, we have
	\begin{align*}
		E_Q\Big[\sup_{s \in [0, \gamma^N_t\wedge \tau^N]} |X(s)-Y(s)|^2\Big]   
		& \le  4E_Q\Big[\int_0^{\gamma^N_t\wedge\tau^N}|X(s)-Y(s)|^2\, dk^N(s) 
		 + \|\sigma\|_\infty^2 \int_0^{\gamma^N_t\wedge\tau^N}|q(s)|^2ds\Big] 
		\\ 
		&\quad + \varepsilon E_Q\Big[\sup_{s \in [0, \gamma^N_t\wedge \tau^N]}|X(s) - Y(s) |^2 \Big]\\
		&\quad + \frac{6C^2_{BDG} + 2\varepsilon}{\varepsilon}E_Q\Big[\int_0^{\gamma_t^N\wedge\tau^N}|X(s)-Y(s)|^2\,dk^N(s) \Big].
	\end{align*} 
	The time change $t\equiv \gamma^N_s$ gives
	\begin{align*}  
	E_Q\Big[\sup_{s \in [0, \gamma^N_t\wedge \tau^N]}|X(s)-Y(s)|^2\Big]  
		&\le  
		\|\sigma\|_\infty^2 E_Q\Big[\int_0^{\gamma^N_t\wedge\tau^N}|q(s)|^2ds\Big] 
		\\ 
		& \quad + \varepsilon E_Q\Big[\sup_{s \in [0, \gamma^N_t\wedge \tau^N]}|X(s) - Y(s) |^2 \Big]\\
		&\quad + 6\frac{C^2_{BDG}+\varepsilon}{\varepsilon}E_Q\Big[\int_0^{t}
		\sup_{r \in [0, \gamma^N_s\wedge \tau^N]}|X(s)-Y(s)|^2\,ds \Big].
	\end{align*}
	Choosing $\varepsilon<1$ then using Gronwall's lemma, we get 
	\begin{align*}
		E_Q\Big[\sup_{s \in [0, \gamma^N_t\wedge \tau^N]}|X(s)-Y(s)|^2\Big] 
		&\le \frac{1}{1-\varepsilon}\|\sigma\|_\infty^2 E_Q\Big[\int_0^{T}|q(s)|^2ds\Big]
		\exp\Big(6\frac{C^2_{BDG}+\varepsilon}{\varepsilon(1 - \varepsilon)}T \Big)
	\end{align*}
	where we also used the fact that $\tau^N\wedge \gamma^N_t\le T$. 
	Letting successively $t$ then $N$ go to infinity, it follows by using Fatou's lemma, $\gamma_t^N\uparrow \infty$, $\tau^N\uparrow T$ and the continuity of $X$ and $Y$ that 
	\begin{align*}
		E_Q\Big[\sup_{s \in [0,T]}|X(s)-Y(s)|^2\Big] &\le \exp\Big(6\frac{C^2_{BDG}+\varepsilon}{\varepsilon(1 - \varepsilon)}T \Big)\frac{1}{1-\varepsilon}\|\sigma\|_\infty^2 E_Q\Big[\int_0^{T}|q(s)|^2ds\Big].
	\end{align*}
	Hence, we conclude from \eqref{eq:identity H} and \eqref{eq:first estim wasser} that
	\begin{equation*}
		\mathcal{W}^2_2(\mu_x,\nu) \le 2\exp\Big(6\frac{C^2_{BDG}+\varepsilon}{\varepsilon(1 - \varepsilon)}T \Big)\frac{1}{1-\varepsilon}\|\sigma\|_\infty^2H(\nu|\mu_x).
	\end{equation*}	
	This concludes the proof.
\end{proof}  

\begin{remark}
	As it appears from the proof, it is conceivable that the above lemma extends to functions $\sigma$ in weighted Sobolev spaces when the set of smooth functions with compact support is dense.
	We restrict ourselves to $W^{0,1}_{2(d+1),\mathrm{loc}}([0,T])$ since this space is enough for our purpose and to simplify the presentation. 
	\end{remark} 

\subsection{Proof of Theorem \ref{thm:t2 sde}} 

We start by the case where condition (A) is fulfilled. 
By \cite[Theorem 3.1]{Bah99} equation \eqref{eq:SDE intro} admits a unique strong solution. 
As in  \cite {Bah99}, the idea consists in using Zvonkin's transform in order to transform equation \eqref{eq:SDE intro} into an SDE without drift then using Proposition \ref{prop:sde no drift} to conclude. 
In the rest of the paper, we denote by $\calL$ the differential operator defined by
\begin{equation*}
	\calL \phi:= b \partial_x \phi + \frac12 tr(\sigma^*\sigma \partial_{xx} \phi).
\end{equation*} 
According to \cite[Theorem 2]{Zvonkin}, there exists a $T >0$ small enough such that the PDE
\begin{equation*}
\begin{cases}
\partial_t\varphi + \calL \varphi  = 0\\
\varphi(T,x) = x
\end{cases}
\end{equation*}
	admits a unique solution $\varphi$ such that:  for every $t$, the function $x\mapsto \varphi(t, x)$ is one-to-one from $\mathbb{R}^d$ onto $\mathbb{R}^d$, both $\varphi$ and its inverse$\psi$ belong to  $W^{1,2}_{p,\mathrm{loc}}([0,T])$ for each $p > 1$,   both $\varphi(t,\cdot)$ as well as its inverse $\psi(t,\cdot)$  are Lipschitz continuous, with Lipschitz constants depending on $d,T, \|b\|_\infty$ and $\|\sigma\|_\infty$.  

Applying It\^o-Krylov's formula, see \cite[Theorem 2.10.1]{Krylov} to $\varphi(t, X_t) :=Y_t$, it follows that $Y$ satisfies the drift-less SDE
\begin{equation*}
Y_t = Y_0 + \int_0^t \tilde\sigma(s, Y_s)\,dW_s
\end{equation*}
with $\tilde \sigma(t,x):= (\sigma^*\partial_x \varphi) (t,\psi(t,y))$. 
 Since $\sigma$ belongs to $W^{0,1}_{2(d+1),\mathrm{loc}}([0,T])$, it follows that 
  $\varphi$ belongs to 
 $ W^{1,2}_{p,\mathrm{loc}}([0,T])$  for each $p>1$ and both  $\varphi$ and $\psi$ are Lipschitz,   it follows that  
$\tilde\sigma \in W^{0,1}_{2(d+1),\mathrm{loc}}([0,T])$. 
Hence, by Lemma \ref{prop:sde no drift}, the law $\mu_y$ of Y satisfies $T_2(C)$, where $C$ is the constant in Proposition \ref{prop:sde no drift}.
But $X_t = \psi(t, Y_t)$ and $\psi$ is Lipschitz continuous.
Thus, the result follows from \cite[Lemma 2.1]{Dje-Gui-Wu}.

\vskip 0.4cm 

We now  assume that condition (B) is fulfilled.
We need to introduce the following Banach spaces:
For every $k \ge0$ and $m\ge1$, let $H^k_m:= (I - \Delta)^{-k/2}L^m$ be the usual space of Bessel potentials on $\mathbb{R}^d$ and denote
\begin{equation*}	
	\mathbb{H}^{2,q}_p([0,T]) := L^q([0,T], H^2_p)\quad \text{and}\quad H^{2,q}_{p}([0,T]):= \{u:[0,T]\to H^2_p\text{ and } \partial_tu \in \Lpq([0,T])\}.
\end{equation*}
The space $H^{2}_p$ is equipped with the norm
\begin{equation*}
	\| u \|_{H^2_p}:= \| (I - \Delta)u \| _{L^p}
\end{equation*}	
making it isomorphic to the Sobolev space $W^2_p(\mathbb{R}^d)$.

Under assumption (B), the existence and uniqueness of $X$ follow e.g. from \cite[Theorem 1.1]{XichZhang11}.
We now show that the law $\mu_x$ of $X$ satisfies $T_2(C)$ for some $C$.
Let $C_b$ be a constant to be determined later.
	By \cite[Theorem 10.3 and Remark 10.4]{KR05}, the PDE
	\begin{equation*}
		\begin{cases}
		\partial_tu^i + \calL u^i + \frac{b^i}{1 +C_b} = 0\\
		u^i(T,x) = 0\end{cases}
	\end{equation*}
	admits a unique solution $u^i \in H^{2,q}_p([0,T])$ and this solution satisfies
	\begin{equation*}
	 	||\partial_t u^i||_{\Lpq} + ||u^i||_{\Hpq} \le \frac{C_1}{1+C_b} ||b^i||_{\Lpq }
	 \end{equation*} 
	 for some constant $C_1$ depending on $d, p,q,T$ and $||b||_{\Lpq}$.
	 Furthermore, since $d/p + 2/q<1$, it follows by \cite[Lemma 10.2]{KR05} that 
	 
	 \begin{equation*}
	 	|\partial_xu^i| \le C_2T^{-1/q}\left(||u^i||_{\Hpq} + T||\partial_tu^i||_{\Lpq}^{} \right)
	 \end{equation*}
	 with $\delta \in (0,1]$ such that $2\delta + \frac dp + \frac2q <2$, and $C_2$ a constant depending on $p,q$ and $\delta$. 
	 Therefore, it holds that
	 
	 \begin{equation}
	 \label{eq:grad estim u lp}
	 	|\partial_xu^i| \le C_1C_2T^{-1/q}(T + 1)\frac{1}{1 + C_b}||b^i||_{\Lpq} \le \frac{C_b}{1 + C_b}
	 \end{equation}
	 with the choice $C_b:= C_1C_2T^{-1/q}(T + 1)\max_{i\in \{1, \dots,d\}}||b^i||_{\Lpq}$.
	 Now consider the function $\Phi^i(t,x):={x^i} + u^i(t,x) $, $i = 1,\dots, d$.
	It is easily checked that the function $\Phi^i$ solves the PDE
	 \begin{equation}
	 \label{eq:pde phi}
		\begin{cases}
	 		\partial_t\Phi^i + \calL \Phi^i = 0\\
	 		\Phi^i(T,x) = {x^i}.
	 	\end{cases}
	 \end{equation}
	 Put $\Phi(t, x)= (\Phi^1(t,x),\dots, \Phi^d(t,x))$.
	 Due to \eqref{eq:grad estim u lp}, it holds that
	 \begin{equation*}
	 	\frac{1}{1 + C_b} |x - y| \le |\Phi(x) - \Phi(y)| \le \frac{1 + 2C_b}{1 + C_b}|x - y|\quad \text{for all } x,y \in \mathbb{R}^d.
	 \end{equation*}
	 As a consequence, $\Phi$ is one-to-one, (see e.g. the corollary on page 87 of \cite{John}), and its inverse $\Psi := \Phi^{-1}$ is $\frac{1}{1+C_b}$-Lipschitz continuous.

	Since for every $t$, $u(t,\cdot)$ belongs to $H^{2}_{p}$, then it can be seen as an element of $W^2_{p}(\mathbb{R}^d)$. 
	Moreover, the derivative of $u$ with respect to $t$ belongs to $L^p$, it thus follows that $u$ belongs to $W^{1,2}_{p}([0,T])$. Hence, the function $\Phi(t,x):={x} + u(t,x)$  belongs to $W^{1,2}_{p, loc}([0,T])$.  Itô-Krylov's formula applied to $\Phi$ gives 
	\begin{align*}
	  	Y_t :=\Phi(t,X_t) &= \Phi(0,x) + \int_0^t(\partial_t\Phi + \calL \Phi)(s,X_s)\,ds + \int_0^t\partial_x\Phi (s,X_s)\sigma\,dW_s\\
	  	&= \Phi(0,x) + \int_0^t\tilde\sigma(s, Y_s)\,dW_s
	  \end{align*}
	with $\tilde\sigma(t,y) := (\sigma^*\partial_x \Phi)(t, \Psi(t,y))$, and where the second equation follows by \eqref{eq:pde phi}.
	
	The rest of the proof follows as in the case of assumption (A).
\hfill$\Box$

\section{Examples}
Let us now present a few examples of multidimensional diffusion models with non-Lipschitz coefficients which fit to our framework.
\subsection{Particles interacting through their rank}
Let $W^1, \dots W^n$ be $n$ independent Brownian motions.
Rank-based interaction models are given by
\begin{equation*}
	dX^{i,n}(t) = \sum_{j=1}^n\delta_j1_{\{X^{i,n}(t) = X^{(i),n}(t)\}}\,dt  + \sigma^i(t)\,dW^i(t)\quad X^{i,n}(0) = x^i
\end{equation*}
for some real numbers $\delta_j$, some measurable, bounded functions $\sigma^i$,   with
$X^{(1),n}(t)\le X^{(2),n}(t)\le \dots\le X^{(n), n}(t)$ is the system in increasing order.
More generally, this model can be written as
\begin{equation*}
	dX^{i,n}(t) = b\Big(\frac1n\sum_{j=1}^n1_{\{X^{n,j}(t)\le X^{n,i}(t)\}}\Big)\,dt + \sigma^i(t)\,dW^i(t)\quad X^{i,n}(0) = x^i
\end{equation*}
for a given (deterministic) functions $b$.
This model was introduced by \citet{Fern-Kara09} in the context of stochastic portfolio theory.
Concentration of measures results for such systems can be found in \cite{Pal-Shkol14}.
When $0<c\le \inf_{i,t}|\sigma^i(t)|\le \sup_{i,t}|\sigma^i(t)|\le C$ for some $c, C$ and $b \in L^\infty$ or $b\in L^p(\mathbb{R},dx)$ (with appropriate $p,d$), our main result shows that the law of $(X^{1, n}, \dots, X^{n,n})$ satisfies $T_2(C)$ for some $C>0$.
This result is also valid for the so-called (finite) Atlas model of \cite{Ban-Fer-Kar05} given by
\begin{equation*}
	dX^{i,n}(t) = \sum_{j=1}^n\delta1_{\{X^{i,n}(t) = X^{p_i,n}(t)\}}\,dt  + \sigma^i(t)\,dW^i(t) \quad X^{i,n}(0) = x^i,
\end{equation*}
for some constant $\delta$ and a permutation $(p_1, \dots, p_n)$ of $(1, \dots, n)$.

\subsection{Particles in quantile interaction}
Quantile interaction models are given by
\begin{equation*}
	dX^{i,n}(t) = b(t,X^{n,i}(t),V^{\alpha,n}(t))\,dt + \sigma(t,X^{n,i}(t))\,dW^i(t)  \quad X^{i,n}(0) = x^i,
\end{equation*}
where $V^{\alpha,n}(t)$ is the quantile at level $\alpha  \in [0,1]$ of the empirical measure of the system\\ $(X^{1,n}(t), \dots, X^{n,n}(t))$.
That is,
\begin{equation*}
	V^{\alpha, n}(t) := \inf\Big\{u \in \mathbb{R}: \frac1n\sum_{i =1}^n1_{\{X^{i,n}(t)\le u\}}\ge \alpha \Big\}.
\end{equation*} 
This model is considered for instance in \cite{Cri-Kur-Lee14} in connection to exchangeable particle systems.
Theorem \ref{thm:t2 sde} can be applied to this case under integrability conditions on $b$ and mild regularity conditions $\sigma$.

\subsection{Brownian motion with random drift}
In addition to particle systems, our main result can also allow to derive transportation inequalities for semimartingales.
We illustrate this in the next corollary.
Let $g$ be a progressive stochastic process.
We call Brownian motion with drift the process
\begin{equation}
\label{eq:bm with drift}
	X(t) = x + \int_0^tg(s)\,ds + \sigma W(t).
\end{equation}
We have the following corollary of Theorem \ref{thm:t2 sde}:
\begin{corollary}
	Assume that the constant matrix $\sigma$ satisfies \eqref{eq:elliptic}.
	If the drift $g$ is bounded and $T$ small enough, then the law $\mu_t^x$ of $X_t$ given by \eqref{eq:bm with drift} satisfies $T_2(C)$ for some $C>0$ depending on $T, \sigma, d$ and $\|g\|_\infty$. 
\end{corollary}
\begin{proof}
	Consider the Borel measurable function
	\begin{equation*}
		b(t, x): = E[g(t)| X(t)=x ].
	\end{equation*}
	By \cite[Corollary 3.7]{Bru-Sch03}, we have $\mu_t^x = \tilde \mu_t$, where $\tilde \mu_t$ is the law of the weak solution $\tilde X_t$ of the SDE 
	\begin{equation}
	\label{eq:sde bounded drift}
		\tilde X(t) = x + \int_0^tb(s, \tilde X(s))\,ds + \sigma W(t).
	\end{equation}
	Since $g$ is bounded so is the function $b$.
	Thus, the SDE \eqref{eq:sde bounded drift} admits a unique strong solution, see e.g. \cite{Bah99} or \cite{Ver81}.
	Thus, $\tilde X$ is necessarily a strong solution and by Theorem \ref{thm:t2 sde} $\tilde \mu$ satisfies $T_2(C)$, which concludes the argument.
\end{proof}
\bibliographystyle{abbrvnat}
 \bibliography{bib_t2bsde}

 \vspace{1cm}

\noindent Khaled Bahlali: IMATH, Universit\'e de Toulon, EA 2134, 83957 La Garde Cedex, France.\\
 {\small\textit{E-mail address:} : khaled.bahlali@univ-tln.fr.}\\

\noindent Soufiane Mouchtabih: LIBMA, Department of Mathematics, Faculty of Sciences Semlalia, Cadi Ayyad University, 2390 Marrakesh, Morocco.\\
and\\
IMATH, Universit\'e  de Toulon, EA 2134, 83957 La Garde Cedex, France.\\
Acknowledgment: supported by PHC Toubkal 18/59.\\
{\small\textit{E-mail address:} : soufiane.mouchtabih@gmail.com}\\

 \noindent Ludovic Tangpi: Department of Operations Research and Financial Engineering, Princeton University, Princeton, 08540,
 NJ; USA.
 \\
 {\small\textit{E-mail address:} ludovic.tangpi@princeton.edu}\\

\end{document}